\documentclass[letterpaper,12pt]{article}

\usepackage{amsmath,amssymb,amscd,latexsym,dsfont,wasysym,bbm}
\usepackage{float}
\usepackage{color}
\usepackage{graphics}
\usepackage{graphicx}
\usepackage{subfigure}
\usepackage{multirow,multicol} 
\usepackage{verbatim}
\usepackage{psfrag}
\usepackage{comment}
\usepackage{makeidx}
\usepackage[]{algorithm}
\usepackage{algorithmicx}
\usepackage{algpseudocode}
\usepackage{cases}
 \usepackage{setspace}
\usepackage{pgfplots} 
\newcommand{\tr}[1]{\textrm{#1}}
\newcommand{\mr}[1]{\mathrm{#1}}
\newcommand{\tnr}[1]{{\textnormal{#1}}}

\newcommand{\mc}[1]{\mathcal{#1}}
\newcommand{\mf}[1]{\mathsf{#1}}
 
\newcommand{\ms}[1]{\mathds{#1}}

\newcommand{\bH}{\boldsymbol{H}}

\newcommand{\bk}{\boldsymbol{k}}

\newcommand{\balpha}{\boldsymbol{\alpha}}

\newcommand{\bdelta}{\boldsymbol{\delta}}
\newcommand{\bepsilon}{\boldsymbol{\epsilon}}

\newcommand{\btheta}{\boldsymbol{\theta}}

\newcommand{\bxi}{\boldsymbol{\xi}}

\newcommand{\ie}{i.e.,~} 		
\newcommand{\eg}{e.g.,~}	

\newcommand{\argmax}{\mathop{\mr{argmax}}}

\newcommand{\set}[1]{\{#1\}}
\newcommand{\SET}[1]{\left\{#1\right\}}

\newcommand{\cd}{\cdot}
\newcommand{\ld}{\ldots}

\newcommand{\e}{\mr{e}}

\newcommand{\PR}[1]{\Pr\SET{#1}}       	

\newcommand{\Ex}{\ms{E}}     			
\newcommand{\Var}{\ms{V}\tr{ar}}     			

\newcommand{\mcN}{\mc{N}}

\newcommand{\mcY}{\mc{Y}}

\newcommand{\mfP}{\mf{P}}

\usepackage[title]{appendix}

\usepackage{dgjournal}   
\usepackage[authoryear,comma,round]{natbib}

\usepackage{amsthm}

\pgfplotsset{compat=1.12}
\usetikzlibrary{positioning,shadows,shapes,fit,arrows,backgrounds}

\tikzstyle{rect_my} = [draw, rectangle, minimum width=2cm, text width=1.8cm, fill=gray!15, 
  text centered,  minimum height=.9cm]
\tikzstyle{square_my} = [draw, rectangle, minimum width=1cm, text width=0.8cm, fill=gray!15, 
  text centered,  minimum height=.9cm]
\tikzstyle{square_my_graph} = [draw, rectangle, minimum width=1.2cm, text width=1cm, fill=gray!15, 
  text centered,  minimum height=1.2cm]
\tikzstyle{circle_my} = [draw, circle, minimum width=1cm, text width=0.8cm, fill=gray!15, 
  text centered,  minimum height=.9cm]
\tikzstyle{circle_my_graph} = [draw, circle, minimum width=1.1cm, text width=.8cm, fill=gray!15, 
  text centered]
\tikzstyle{cloud_my} = [draw, shape=cloud, minimum width=1cm, text width=0.8cm, fill=gray!15, 
  text centered,  minimum height=1cm]

\tikzstyle{point_my} = [draw=none, minimum width=0cm, text width=0cm, fill=none, 
  text centered,  minimum height=0cm]    
\tikzstyle{line_my} = [draw, -latex]    
\tikzstyle{box_my}=[draw, minimum size=2em, text width=4.5em, text centered]
\tikzstyle{bigbox_my}=[draw, inner sep=15pt]
\tikzstyle{arrow_my} = [thick,->,>=stealth]
\tikzstyle{noarrow_my} = [thick,-,=>stealth]

\usepackage[acronym,nonumberlist]{glossaries} 
\usepackage{url}
\usepackage{enumitem}
\setlist[itemize]{leftmargin=*,topsep=2pt,itemsep=2pt,parsep=0pt}
\makeatletter\g@addto@macro\UrlBreaks{\do\-}\makeatother 

\newacronym[\glsshortpluralkey=PDFs,\glslongpluralkey=probability density functions]{pdf}{PDF}{probability density function}
\newacronym[\glsshortpluralkey=CDFs,\glslongpluralkey=cumulative density functions]{cdf}{CDF}{cumulative density function}
\newacronym[\glsshortpluralkey=CCDFs,\glslongpluralkey=complementary cumulative density functions]{ccdf}{CDF}{complementary cumulative density function}
\newacronym[\glsshortpluralkey=PMFs,\glslongpluralkey=probability mass functions]{pmf}{PMF}{probability mass function}
\newacronym[]{lhs}{l.h.s.}{left-hand side}
\newacronym[]{rhs}{r.h.s.}{right-hand side} 

\newacronym[]{bicm}{BICM}{bit-interleaved coded modulation}
\newacronym[]{bicmid}{BICM-ID}{BICM with iterative demapping}
\newacronym[]{cm}{CM}{coded modulation}
\newacronym[]{tcm}{TCM}{trellis-coded modulation}
\newacronym[]{mlc}{MLC}{multi-level coding}
\newacronym[]{pam}{PAM}{pulse amplitude modulation}
\newacronym[]{bpsk}{BPSK}{binary phase shift keying}
\newacronym[]{qam}{QAM}{quadrature amplitude modulation}
\newacronym[]{16qam}{16-QAM}{16-points quadrature amplitude modulation}
\newacronym[]{psk}{PSK}{phase shift keying}
\newacronym[\glsshortpluralkey=LLRs,\glslongpluralkey=logarithmic likelihood ratios]{llr}{LLR}{logarithmic likelihood ratio}
\newacronym[]{oc}{OC}{operating characteristic}

\newacronym[\glsshortpluralkey=MIs,\glslongpluralkey=mutual informations]{mi}{MI}{mutual information}
\newacronym[\glsshortpluralkey=GMIs,\glslongpluralkey=generalized mutual informations]{gmi}{GMI}{generalized mutual information}
\newacronym[]{eesm}{EESM}{exponential effective-SNR-mapping}
\newacronym[]{bicm-gmi}{BICM-GMI}{BICM generalized mutual information}
\newacronym[]{awgn}{AWGN}{additive white Gaussian noise}
\newacronym[]{bsc}{BSC}{binary symetric channel}
\newacronym[]{amc}{AMC}{adaptive modulation and coding}
\newacronym[]{csi}{CSI}{channel state information}
\newacronym[]{cqi}{CQI}{channel quality indicator}
\newacronym[]{kl}{KL}{Kullback-Leibler}
\newacronym[]{cmm}{CMM}{circular moment matching}
\newacronym[]{ga}{GA}{Gaussian approximation}

\newacronym[]{sp}{SP}{set-partitioning}
\newacronym[]{gsm}{GSM}{global system for mobile communications}
\newacronym[]{edge}{EDGE}{enhanced data rates for GSM evolution}
\newacronym[]{3gpp}{3GPP}{3rd generation partnership project}
\newacronym[]{umts}{UMTS}{Universal Mobile Telecommunication System}
\newacronym[]{lte}{LTE}{Long Term Evolution}
\newacronym[]{dvb}{DVB}{digital video broadcasting}
\newacronym[]{fdd}{FDD}{Frequency Division Duplexing}

\newacronym[\glsshortpluralkey=CCs,\glslongpluralkey=convolutional codes]{cc}{CC}{convolutional code}
\newacronym[\glsshortpluralkey=PCCCs,\glslongpluralkey=parallel concatenated convolutional codes]{pccc}{PCCC}{parallel concatenated convolutional code}
\newacronym[\glsshortpluralkey=TCs,\glslongpluralkey=turbo codes]{tc}{TC}{turbo code}
\newacronym{ldpc}{LDPC}{low-density parity-check}
\newacronym[]{ofdm}{OFDM}{orthogonal frequency-division multiplexing}
\newacronym[]{bep}{BEP}{bit-error probability}
\newacronym[]{wep}{WEP}{word-error probability}
\newacronym[]{sep}{SEP}{symbol-error probability}
\newacronym[]{pep}{PEP}{pairwise-error probability}
\newacronym[]{ttcm}{TTCM}{turbo-trellis coded modulation}
\newacronym[]{uep}{UEP}{unequal error protection}
\newacronym[\glsshortpluralkey=CENCs,\glslongpluralkey=convolutional encoders]{cenc}{CENC}{convolutional encoder}
\newacronym[]{mimo}{MIMO}{multiple-input multiple-output}
\newacronym[\glsshortpluralkey=SNRs,\glslongpluralkey=signal-to-noise ratios]{snr}{SNR}{signal-to-noise ratio}
\newacronym[\glsshortpluralkey=SINRs,\glslongpluralkey=signal-to-interference-plus-noise ratios]{sinr}{SINR}{signal-to-interference-plus-noise ratio}
\newacronym[]{msb}{MSB}{most-significative bit}
\newacronym[]{bcjr}{BCJR}{Bahl--Cocke--Jelinek--Raviv}
\newacronym[]{cbc}{CBC}{Colavolpe--Barbieri--Caire}
\newacronym[]{skr}{SKR}{Shayovitz--Kreimer--Raphaeli}
\newacronym[\glsshortpluralkey=SEDs,\glslongpluralkey=squared Euclidean distances]{sed}{SED}{squared Euclidean distance}
\newacronym[\glsshortpluralkey=EDs,\glslongpluralkey=Euclidean distances]{ed}{ED}{Euclidean distance}
\newacronym[\glsshortpluralkey=MEDs,\glslongpluralkey=minimum Euclidean distances]{med}{MED}{minimum Euclidean distance}
\newacronym[]{core}{CoRe}{constellation rearrangement}
\newacronym[]{pdl}{PDL}{parallel decoding of the individual levels}
\newacronym[\glsshortpluralkey=GCs,\glslongpluralkey=Gray codes]{gc}{GC}{Gray code}
\newacronym[]{brgc}{BRGC}{binary-reflected Gray code}
\newacronym[]{nbc}{NBC}{natural binary code}
\newacronym[]{fbc}{FBC}{folded-binary code}
\newacronym[]{bsgc}{BSGC}{binary semi-Gray code}
\newacronym[]{msp}{MSP}{modified set-partitioning}
\newacronym[]{ssp}{SSP}{semi set-partitioning}
\newacronym[]{fhd}{FHD}{free Hamming distance}
\newacronym[]{mfhd}{MFHD}{maximum free Hamming distance}
\newacronym[]{ods}{ODS}{optimal distance spectrum}
\newacronym[]{iud}{i.u.d.}{independent and uniformly distributed}
\newacronym[]{ud}{u.d.}{uniformly distributed}
\newacronym[]{iid}{i.i.d.}{independent, identically distributed}
\newacronym[]{ami}{AMI}{accumulated mutual information}
\newacronym[]{bico}{BICO}{binary-input continuous-output}
\newacronym[]{gh}{GH}{Gauss--Hermite}
\newacronym[]{gum}{GUM}{Gaussian--uniform mixture}

\newacronym[\glsshortpluralkey=BSs,\glslongpluralkey=base-stations]{bs}{BS}{base-station}
\newacronym[\glsshortpluralkey=MSs,\glslongpluralkey=mobile-stations]{ms}{MS}{mobile-stations}

\newacronym[]{phy}{PHY}{physical layer} 
\newacronym[]{rlc}{RLC}{Radio-Link control} 
\newacronym[]{ran}{RAN}{Radio Access Network} 
\newacronym[]{llc}{LLC}{logical link control} 
\newacronym[]{tcp}{TCP}{transmission control protocol} 
\newacronym[]{mac}{MAC}{media access control} 
\newacronym[]{fft}{FFT}{fast Fourier transform} 
\newacronym[]{ft}{FT}{Fourrier transform}
\newacronym[]{cf}{CF}{characteristic function} 
\newacronym[]{mgf}{MGF}{moment generating function} 
\newacronym[]{ee}{EE}{energy efficiency} 
\newacronym[]{eb}{EB}{energy per bit}
\newacronym[]{kkt}{KKT}{Karush--Kuhn--Tucker} 
\newacronym[]{mcs}{MCS}{modulation/coding scheme} 
\newacronym[]{fec}{FEC}{forward error correction}
\newacronym[]{arq}{ARQ}{automatic repeat request}
\newacronym[]{harq}{HARQ}{hybrid ARQ}
\newacronym[]{tarq}{TARQ}{truncated HARQ}
\newacronym[]{ir}{IR}{incremental redundancy}
\newacronym[]{rpr}{RR}{repetition redundancy}
\newacronym[]{rrharq}{RR-HARQ}{repetition redundancy HARQ}
\newacronym[]{irharq}{IR-HARQ}{incremental redundancy HARQ}
\newacronym[]{ack}{ACK}{positive acknowledgment}
\newacronym[]{nack}{NACK}{negative acknowledgment}
\newacronym[]{hol}{HoL}{head of the line}
\newacronym[]{crc}{CRC}{cyclic redundancy check}
\newacronym[]{dp}{DP}{dynamic programming}
\newacronym[]{gp}{GP}{geometric programming}
\newacronym[]{per}{PER}{packet error rate}
\newacronym[]{ber}{BER}{bit error rate}
\newacronym[]{op}{OP}{outage probability}
\newacronym[]{spa}{SPA}{saddle-point approximation}
\newacronym[]{mrc}{MRC}{maximum ratio combining}
\newacronym[]{mdp}{MDP}{Markov decision process}
\newacronym[]{lp}{LP}{linear programming}
\newacronym[]{pomdp}{POMDP}{partially observable Markov decision process}
\newacronym[]{psimdp}{PSI-MDP}{partial state information Markov decision process}
\newacronym[]{scpp}{SCPP}{stochastic shortest path problem}

\newacronym[]{forw}{frwd}{forward}
\newacronym[]{feed}{fdbk}{feedback}

\newacronym[]{mm}{MM-HARQ}{multi-message HARQ}
\newacronym[]{xp}{XP-HARQ}{cross-packet HARQ}
\newacronym[]{ts}{TS}{time-sharing}
\newacronym[]{sc}{SC}{superposition coding}
\newacronym[]{sbrq}{SBRQ}{systematic backward retransmission}
\newacronym[]{brq}{BRQ}{backward retransmission}
\newacronym[]{lharq}{L-HARQ}{layer-coded HARQ}
\newacronym[]{anlharq}{AoN-HARQ}{all-or-none L-HARQ}
\newacronym[]{vlharq}{VL-HARQ}{variable-length HARQ}

\newacronym[]{pp}{PPP}{point process}
\newacronym[]{ppp}{PPP}{Poisson point process}

\newacronym[]{fide}{FIDE}{F\'ed\'eration Internationale des \'Echecs}
\newacronym[]{fifa}{FIFA}{F\'ed\'eration Internationale de Football Association}
\newacronym[]{fivb}{FIVB}{F\'ed\'eration Internationale de Volleyball}
\newacronym[]{epl}{EPL}{English Premier League}
\newacronym[]{nhl}{NHL}{National Hockey League}
\newacronym[]{shl}{SHL}{Swedish Hockey League}
\newacronym[]{nfl}{NFL}{National Football League}
\newacronym[]{ipl}{IPL}{Indian Premier League}
\newacronym[]{nba}{NBA}{National Basketball Association}
\newacronym[]{mls}{MLS}{Major League Soccer}

\newacronym[]{sg}{SG}{stochastic gradient}
\newacronym[]{lms}{LMS}{least mean squares}
\newacronym[]{nlms}{NLMS}{normalized LMS}
\newacronym[]{rls}{RLS}{recursive least squares}
\newacronym[]{vss}{VSS}{variable step-size}
\newacronym[]{hfa}{HFA}{home-field advantage}
\newacronym[]{ha}{HA}{home advantage}
\newacronym[]{mov}{MOV}{margin of victory}
\newacronym[]{ac}{AC}{adjacent categories}
\newacronym[]{cl}{CL}{cumulative link}
\newacronym[]{ci}{CI}{confidence interval}
\newacronym[]{glm}{GLM}{generalized linear models}
\newacronym[]{nn}{NN}{neural networks}
\newacronym[]{rps}{RPS}{ranked probability score}
\newacronym[]{mse}{MSE}{mean square error}
\newacronym[]{mmse}{MMSE}{minimum mean square error}
\newacronym[]{rmse}{RMSE}{root mean squared error}
\newacronym[]{ols}{OLS}{ordinary least squares}
\newacronym[]{map}{MAP}{maximum a posteriori}
\newacronym[]{ml}{ML}{maximum likelihood}
\newacronym[]{loo}{LOO}{leave-one-out}
\newacronym[]{alo}{ALO}{approximate leave-one-out}
\newacronym[]{logo}{LOGO}{leave-one-game-out}
\newacronym[]{alogo}{ALOGO}{approximate leave-one-game-out}
\newacronym[]{msd}{MSD}{mean-square deviation}
\newacronym[]{lop}{LOP}{linear ordering problem}
\newacronym[]{so}{SO}{shootouts}
\newacronym[]{rt}{RT}{regulation time}
\newacronym[]{ot}{OT}{overtime}
\newacronym[]{rr}{RR}{round-robin}
\newacronym[]{irt}{IRT}{item-response theory}

\newacronym[]{dmp}{DMP}{discretized message passing}
\newacronym[]{mp}{MP}{message passing}
\newacronym[]{ep}{EP}{expectation propagation}
\newacronym[]{em}{EM}{expectation maximization}
\newacronym[]{hmm}{HMM}{hidden Markov models}

\newacronym[]{svd}{SVD}{singular values decomposition}

\newacronym[]{vkf}{vKF}{\emph{vector-variance} Kalman Filter}
\newacronym[]{skf}{sKF}{\emph{scalar-variance} Kalman Filter}
\newacronym[]{fkf}{fKF}{\emph{fixed-variance} Kalman Filter}
\newacronym[]{kf}{KF}{Kalman filter}
\newacronym[]{gelo}{G-Elo}{generalized Elo}
\newacronym[]{mvdr}{MVDR}{minimum variance distortionless response}
\newacronym[]{lcmv}{LCMV}{linearly-constrained minimum variance}
\newacronym[]{music}{MUSIC}{multiple signal classification}
\newacronym[]{cp}{CP}{canonical polyadic}

\newacronym[]{tpb}{TPB}{tensor-product-basis}

\newtheorem{corollary}{Corollary}

\newtheorem{proposition}{Proposition}
\newtheorem{lemma}{Lemma}

\begin{document}

\title{
Ranking by points and ordinal models
}
\author{Leszek Szczecinski
\thanks{%
L.~Szczecinski  is with Institut National de la Recherche Scientifique, Montreal, Canada. [e-mail: Leszek.Szczecinski@inrs.ca].}%
}%

\originalmaketitle  

\setstretch{1.6} 

\begin{abstract}
Sport leagues assign a fixed score-point to each match outcome and rank the players (or teams) by their accumulated scores. Since the score-points are mandated without a probabilistic model in mind, we investigate models that make the score the right summary of the data. We find the score to be a sufficient statistic if and only if the points rule is constant-sum, so that every match distributes a fixed number of points between its participants, and the outcomes follow the adjacent categories model whose slope parameters are the score-points themselves. The constant-sum condition fails in the National Hockey League and in football, which awards three points for a win instead of the two required by sufficiency. We also establish that when two players meet every other opponent equally often, and each plays as often at home as away, they are ordered identically by the score and by the estimated skills; a venue-balanced round-robin satisfies the condition for all pairs. The Elo ranking then produces the same order. In nine leagues from football, ice hockey and volleyball, the two orders differ only where the schedule condition fails, and the uniform rule awarding one point per outcome level agrees with the estimated skills more closely than each league's own rule, unless that rule is already the uniform one.
\end{abstract}

\begingroup\setstretch{1}%
\par\endgroup

\section{Introduction}\label{Sec:Intro}

Rankings in one-on-one competitions such as sport leagues or tournaments are almost universally produced by \emph{counting points}: each of the $L$ possible outcomes $y$ is assigned a fixed score-point $\xi_y$ and the players (or teams) are ordered by their accumulated \emph{scores} $s_i=\sum_y\xi_y k_{y;i}$ \citep{Langeville12_book}, where $k_{y;i}$ counts how often the player $i$ obtained the outcome $y$. The outcomes are ordinal, so the one indexed by the host as $y$ is indexed by the visitor as $L-1-y$, giving them, respectively, $\xi_y$ and $\xi_{L-1-y}$ score-points. The score-points are mandated, \eg $(0\text{-}1\text{-}3)$ in association football or $(0\text{-}1\text{-}2\text{-}2)$ in ice hockey, and no model relates them to the skills, so the score $s_i$ has no probabilistic interpretation, and yet the order it induces is routinely taken to define what it means to be ``the best''. In parallel, a large body of work fits \emph{probabilistic models} to the observed outcomes and ranks the players by an estimated skill $\theta_i$ \citep{Newman23}. Our objective is to make explicit the link between the two approaches.

We fix the ranking, as used in practice, and ask which model it corresponds to: in the statistical sense, we require the accumulated scores $s_i$ to be a sufficient statistic for the skills. We find that the score is sufficient if and only if the points rule is constant-sum, $\xi_y+\xi_{L-1-y}=\tnr{const}$, and the outcomes follow the \gls{ac} model whose slopes are the score-points themselves. Thus, the condition requires every match to distribute the same total. For binary matches this restates a classical fact, the win count being the sufficient statistic of the Bradley--Terry model \citep{Zermelo29,Bradley52,Ford57}, and the ordinal statement has a counterpart in the item-response literature \citep{Andersen77,Masters82}, where the home-away distinction does not appear and the score-points are, in effect, constrained to be uniform, a restriction we do not impose.

One consequence is a simple constant-sum test for a points rule to yield a sufficient score, which some familiar rules fail: the football $(0\text{-}1\text{-}3)$ rule distributes three points in a match with a winner and two in a draw; similarly, the \gls{nhl} rule distributes two points in regulation time and three in overtime. With three outcomes, the test is passed when the draw is worth half a win, a value otherwise adopted by convention or identified by simulation \citep{Lasek18}.

We also identify a condition under which ranking by the score reproduces the ranking by the estimated skills, namely that the two players meet every other opponent equally often and each plays as often at home as away. A venue-balanced round-robin satisfies it for all players, which generalizes the results from binary tournaments \citep{Zermelo29,Ford57} to ordinal outcomes and the venue. The order does not depend on the home advantage or on the prior, which explains why estimators as different as the Elo rating and the \gls{map} estimate induce the same ranking \citep{Szczecinski26b}.

Using data from association football, ice hockey and volleyball, we then measure how much the two orders differ where this condition fails. Only a few pairs are reordered even in the \gls{nhl}, whose schedule is the furthest from schedule-equivalent, which explains why point counting works well. Separately, we interpret each league's own points rule as an implicitly mandated model and compare it with the uniform rule, which awards one point per outcome level. The uniform rule needs no estimation and, in every league we examine, outperforms the league's own rule unless that rule is already the uniform one, so we recommend it.

Section~\ref{Sec:Model} defines the ranking by point counting and the ordinal models, and states the main result (Lemma~\ref{Lem:score.characterization}). Section~\ref{Sec:Estimation} derives the \gls{map} estimate and the conditions under which the score recovers the skill order, Sec.~\ref{Sec:Numerical} presents the numerical results, and Sec.~\ref{Sec:Conclusions} concludes. The proof of Lemma~\ref{Lem:score.characterization} is given in the Appendix.

\section{Models and ranking strategies}\label{Sec:Model}

We consider $M$ players (or teams) indexed by $i=1,\ld,M$ facing each other in one-on-one matches. The outcome of a match is $y\in\mcY=\set{0,1,\ld,L-1}$, ordered from the point of view of one of the two players, the outcome $y=l$ being more valuable for that player than $y=l-1$. For $L=3$, $y=0$ is their loss, $y=1$ the draw, and $y=2$ their win.

The elementary count is $k^\tnr{h}_{y;i,j}$: the number of matches \emph{hosted by} $i$ against $j$ that ended in $i$'s outcome $y$; for all pairs it lists every match once, each match having a single host. The same matches, seen from the visitor's side, are counted as
\begin{align}
\label{k.a.def}
    k^\tnr{a}_{y;i,j}:=k^\tnr{h}_{L-1-y;j,i},
\end{align}
\ie those where $i$ visited $j$, and outcome $y$ for $i$ is the outcome $L-1-y$ for $j$. Every match thus has a home entry for its host and an away entry for its visitor; $k^\tnr{a}_{y;i,j}$ is an alias that adds no information and only eases the notation. Throughout, the first index names the player from whose perspective the outcome and the skill difference are taken, the second the opponent, and the superscript records where that first player played; swapping the pair therefore replaces $y$ by $L-1-y$, reverses the sign of the skill difference and exchanges the two venues.

Summing over opponents gives player $i$'s home and away counts, as well as the count of outcomes $y$ obtained by $i$ at either venue
\begin{align}
\label{player.records}
    k^\tnr{h}_{y;i}&=\sum_j k^\tnr{h}_{y;i,j},\qquad
    k^\tnr{a}_{y;i}=\sum_j k^\tnr{a}_{y;i,j},\qquad
    k_{y;i}=k^\tnr{h}_{y;i}+k^\tnr{a}_{y;i}.
\end{align}

Summing over outcomes gives $k^\tnr{h}_{i,j}=\sum_y k^\tnr{h}_{y;i,j}$ which is the number of times $i$ hosted $j$, $k^\tnr{a}_{i,j}=\sum_y k^\tnr{a}_{y;i,j}=k^\tnr{h}_{j,i}$ the number of times $i$ visited $j$, $k_{i,j}=k^\tnr{h}_{i,j}+k^\tnr{a}_{i,j}$ the total number of $i$ vs $j$ matches, and $k^\tnr{h}_i=\sum_j k^\tnr{h}_{i,j}$, $k^\tnr{a}_i=\sum_j k^\tnr{a}_{i,j}$, $k_i=k^\tnr{h}_i+k^\tnr{a}_i$ the home, away, and total match counts of $i$.

We call the pair $(i,j)$ \emph{schedule-equivalent} if $i$ and $j$ meet every other opponent the same number of times, $k_{i,l}=k_{j,l}$ for $l\ne i,j$. 
Requiring it for every pair forces all the counts to be the same,\footnote{Applying it to the pair $(j,l)$ at the opponent $i$ gives $k_{i,j}=k_{i,l}$, so each player meets all opponents equally often, and $k_{i,j}=k_{j,i}$ makes the common value the same for all players.} which is a \emph{round-robin} schedule, in which every pair meets $k\ge1$ times.

We say that the schedule is \emph{venue-balanced for the player $i$} if $i$ hosts and visits every opponent equally often, $k^\tnr{h}_{i,l}=k^\tnr{a}_{i,l}$ for every $l$, and \emph{venue-balanced} if this holds for every player. Venue balance requires an even number of meetings per pair. 

Under the fundamental assumption that the players can be ordered, $\ld \succ i \succ j \succ \ld$, where $i\succ j$ means that $i$ is ``better''/``stronger'' than $j$, the \emph{ranking} is an estimate of this non-observable order. Ranking algorithms most often operate in two steps: a scalar value is assigned to each player, reflecting their ``strength'' or ``merit'' \citep{Langeville12_book,Newman23}, and the players are then sorted by that value. Some algorithms are derived from statistical models and require an optimization to be solved, while others, such as those based on point counting, come from practice.

\subsection{Ranking via point counting}\label{Sec:score.point.accumulation}

The ranking used in sport leagues assigns the \emph{score-point} $\xi_y$ to the result $y$ and calculates the total score\footnote{We use \emph{score} in its sporting sense throughout: the accumulated total $s_i$ by which players are ranked. It should not be confused with the statistical \emph{score}, the gradient of the log-likelihood (Fisher's score function); where the latter is meant we write \emph{gradient}.} for each player over all matches
\begin{align}
\label{score.definition}
    s_i=\sum_y \xi_{y} k_{y;i}.
\end{align}

The ranking is then determined by sorting the scores in descending order
\begin{align}
\label{score.implication}
    s_i> s_j \implies  i \succ j.
\end{align}

The tuple of score-points $\bxi=(\xi_0\text{-}\ld\text{-}\xi_{L-1})$ is called the \emph{points rule}. For example, in the ternary (win/draw/loss) matches, $L=3$, the rule \eqref{score.definition} is summarized by the triplet $(\xi_0\text{-}\xi_1\text{-}\xi_2)$, which means that we assign $\xi_2$ points to the win, $\xi_1$ points to the draw, and $\xi_0$ points to the loss.

An increasing affine transformation of the score-points, $\xi_y\leftarrow a\xi_y +b$ with $a>0$, does not change the order of $s_i$ when every player plays the same number of matches $k_i$.\footnote{The scaling gives $s_i\leftarrow a\,s_i$ and the shift adds $b\sum_y k_{y;i}=b k_i$, a player-independent constant.} We may thus fix two of the score-points, and we use whichever normalization is convenient, such as $(0\text{-}\tfrac{1}{2}\text{-}1)$ or $(0\text{-}1\text{-}2)$.

The ranking \eqref{score.implication} is so common that it is often equated with the very meaning of ``the best'', although the score-points are mandated, \eg by sports federations, and no model relates the skills to the accumulated score. Its pervasiveness indicates however that it conveys information that agrees with the common intuition about the skills. We therefore ask whether there is more to point counting than an arbitrary convention, and look for the model in which the score is the right summary of the data, that is, a \emph{sufficient statistic}. Section~\ref{Sec:AC.model} shows that this requirement identifies the form of the model and makes its parameters equal to the score-points, up to an affine transformation.

\subsection{Outcome models and the venue}\label{Sec:ranking=inference}

The probabilistic approach to ranking assumes that the skills of the players $\theta_i$ and $\theta_j$ determine the probability of the outcome $y$, most often through their difference $z_{i,j}=\theta_i-\theta_j$, and it is common to let the venue matter as well \citep{Kuk95,fifa_rating_W}:\footnote{The motivation to separate the venues appears for the extreme outcomes: for a given skill difference $z_{i,j}$, the win $y=L-1$ should be more probable when $i$ is the host than when $i$ is the visitor. Similarly, the host's loss should be less probable than the visitor's loss.}
\begin{align}
\label{P.h.def}
    \PR{Y=y|\theta_i,\theta_j, i~ \text{is the host}}&=\mfP^\tnr{h}_y(z_{i,j})\\
\label{P.a.def}
    \PR{Y=y|\theta_i,\theta_j, i~ \text{is the visitor}}&=\mfP^\tnr{a}_y(z_{i,j}).
\end{align}
In both, $Y$ is the outcome from $i$'s perspective and the superscript indicates the venue at which $i$ plays. In a match where $i$ hosts $j$, that outcome $y$ corresponds to the outcome $L-1-y$ for $j$. By swapping $i\leftrightarrow j$, the skill difference reverses sign, $z_{j,i} = -z_{i,j}$, and since we refer to the same event, the probabilities must coincide:
\begin{align}
\label{P.h.P.a.reciprocity}
    \mfP^\tnr{h}_y(z_{i,j})=\mfP^\tnr{a}_{L-1-y}(-z_{i,j}),
\end{align}
thus \eqref{P.a.def} follows from \eqref{P.h.def} and we only need one function $\mfP^\tnr{h}_y(z)$.

The log-likelihoods are $\ell^\tnr{h}_y(z)=\log\mfP^\tnr{h}_y(z)$ and $\ell^\tnr{a}_y(z)=\log\mfP^\tnr{a}_y(z)$, and they satisfy the same reciprocity. In the absence of the home advantage the two coincide and we remove the superscript, $\mfP^\tnr{h}_y(z)=\mfP^\tnr{a}_y(z)=\mfP_y(z)$, which means that
\begin{align}
\label{P.y.symmetry}
    \mfP_y(z)=\mfP_{L-1-y}(-z).
\end{align}

\subsection{Point counting identifies the \gls{ac} model}\label{Sec:AC.model}

We require the score to be a \emph{sufficient statistic} for the skills, that is, the data enter the estimating equation of each player $i$ only through that player's score $s_i$, for all skill values and schedules $k^\tnr{h}_{i,j}$. The outcome model may depend on the venue, with $\mfP^\tnr{h}_y(z)$ in \eqref{P.h.def} at home, the away model fixed by the reciprocity~\eqref{P.h.P.a.reciprocity}, and $\mfP^\tnr{h}_y(z)$ not constant in $z$. The points $\xi_y$ in the count~\eqref{score.definition} are non-decreasing and normalized to $\xi_0=0$, $\xi_{L-1}=1$. They are also rational, as any points rule in use is.

\begin{lemma}[Sufficiency of the point count]\label{Lem:score.characterization}
The point count~\eqref{score.definition} is a sufficient statistic for the skills if and only if
(i) the score-points are constant-sum, $\xi_y+\xi_{L-1-y}=1$, an outcome indexed $y$ by the host being indexed $L-1-y$ by the visitor; and
(ii) the outcome probabilities have the form
\begin{align}
\label{AC.model}
    \mfP^\tnr{h}_{y}(z)&= \frac{\e^{\alpha^\tnr{h}_y + \delta_y z}}{D^\tnr{h}(z)},
    \quad D^\tnr{h}(z)=\sum_{l=0}^{L-1} \e^{\alpha^\tnr{h}_l + \delta_l z},
    \quad y=0,\ld, L-1,
\end{align}
with the linear predictor $\alpha^\tnr{h}_y+\delta_y z$, whose slopes are the score-points, $\delta_y=\xi_y$, and whose intercepts $\balpha^\tnr{h}$ are unconstrained.
\end{lemma}

\begin{proof}
See Appendix~\ref{App:proof}.
\end{proof}

The form \eqref{AC.model} is the multinomial logistic \gls{ac} model \citep{Bock72,Tutz20,Egidi21,Szczecinski22},\footnote{The model owes its name to the adjacent-category log-odds, which are linear in $z$, $\log[\mfP^\tnr{h}_{y+1}(z)/\mfP^\tnr{h}_y(z)]=(\alpha^\tnr{h}_{y+1}-\alpha^\tnr{h}_y)+(\delta_{y+1}-\delta_y)z$.} with $D^\tnr{h}(z)$ the normalization factor (\ie a partition function \citep[Sec.~4.2]{Barber12_Book}) and the log-likelihood
\begin{align}
\label{AC.model.ell}
\ell^\tnr{h}_{y}(z)&= \alpha^\tnr{h}_y + \delta_y z -\log D^\tnr{h}(z).
\end{align}
Strictly increasing slopes, $\delta_y<\delta_{y+1}$, $y=0,\ld,L-2$, make the model ordinal \citep{Tutz19}; where a rule gives two outcomes the same score-point, their probabilities keep a fixed ratio and the skill does not separate them. The constant-sum condition then fixes the slope vector as $\bdelta=[0,\delta_1,\ld,1-\delta_1,1]$, with $\lfloor(L-2)/2\rfloor$ free entries: none for $L=3$, where $\delta_1=\tfrac12$ is forced, one for $L=4$, and two for $L=6$. We fix $\alpha^\tnr{h}_0=0$ without loss of generality.\footnote{The model is unchanged under $\alpha^\tnr{h}_l\leftarrow\alpha^\tnr{h}_l-\alpha_\tnr{ref}$, and we use $\alpha_\tnr{ref}=\alpha^\tnr{h}_0$.} 

Differentiating \eqref{AC.model.ell} we obtain
\begin{align}
\label{dot.ell_y(z).AC}
    \dot\ell^\tnr{h}_y(z)&=\delta_y - G^\tnr{h}(z),
\end{align}
where, for a given $z$
\begin{align}
    \label{G(z).AC}
    G^\tnr{h}(z) & =\sum_{l=0}^{L-1}\delta_l\mfP^\tnr{h}_l (z)=\Ex[\delta_Y\mid z]
\end{align}
calculates the expected score from the host's perspective. Differentiating \eqref{AC.model} yields $\dot{\mfP}^\tnr{h}_l(z)=\mfP^\tnr{h}_l(z)\bigl[\delta_l-G^\tnr{h}(z)\bigr]$. Therefore $\dot G^\tnr{h}(z)=\sum_l\delta_l\mfP^\tnr{h}_l(z)\bigl[\delta_l-G^\tnr{h}(z)\bigr]=\Var[\delta_Y\mid z]$, which is positive because the slopes are not all equal, so $G^\tnr{h}(\cd)$ is strictly increasing. The visitor's counterparts are $\ell^\tnr{a}_y(z)$ and $G^\tnr{a}(z)=\sum_{l=0}^{L-1}\delta_l\mfP^\tnr{a}_l(z)$. The reciprocity \eqref{P.h.P.a.reciprocity} gives $\dot\ell^\tnr{a}_y(z)=-\dot\ell^\tnr{h}_{L-1-y}(-z)$ and, with the constant-sum slopes, $G^\tnr{h}(z)=1-G^\tnr{a}(-z)$, so $G^\tnr{a}(\cd)$ increases as well. The visitor's gradient is then of the same form as \eqref{dot.ell_y(z).AC},
\begin{align}
    \label{dot.ell_y(z).AC.a}
    \dot\ell^\tnr{a}_y(z)&=\delta_y - G^\tnr{a}(z).
\end{align}
At both venues the gradient is the outcome's score-point less the expected score.

\paragraph{The home advantage.}
The model \eqref{AC.model} is the same at both venues and only the intercepts differ, $\balpha^\tnr{h}=[0,\alpha^\tnr{h}_1,\ld,\alpha^\tnr{h}_{L-1}]$ for the home and $\balpha^\tnr{a}$ for the away venue. Namely, from $\mfP^\tnr{h}_y(z)=\mfP^\tnr{a}_{L-1-y}(-z)$ in \eqref{P.h.P.a.reciprocity}, we have
\begin{align}
   \mfP^\tnr{a}_y(z) \propto \e^{\alpha^\tnr{h}_{L-1-y}-\delta_{L-1-y} z}\propto \e^{\alpha^\tnr{a}_{y}+\delta_{y} z},\nonumber
\end{align}
where the second step uses the constant-sum $\delta_{L-1-y}=1-\delta_y$ and absorbs the resulting factor $\e^{-z}$, which does not depend on $y$, into the normalization. Thus, only the intercepts $\alpha^\tnr{a}_y=\alpha^\tnr{h}_{L-1-y}$ differentiate between venues.

The venue-neutral model \eqref{P.y.symmetry} imposes a symmetry $\alpha_y = \alpha_{L-1-y}$ \citep[Sec.~3]{Agresti92} and then we drop the superscript and write $\mfP_y(z)$, $\ell_y(z)$, $D(z)$ and $G(z)$. The home advantage is then obtained by boosting the home skill by $\eta$ \citep{Rao67,Davidson77,fifa_rating_W,Lasek20,Szczecinski22a}
\begin{align}
\label{P.h.a.from.P}
    \mfP^\tnr{h}_y(z)=\mfP_y(z+\eta),\qquad \mfP^\tnr{a}_y(z)=\mfP_y(z-\eta),
\end{align}
which satisfies the reciprocity \eqref{P.h.P.a.reciprocity} and, in the \gls{ac} model, is equivalent to shifting the symmetric intercepts $\alpha_y=\alpha_{L-1-y}$,
\begin{align}
\label{alpha.sym->alpha}
	\alpha^\tnr{h}_y &= \alpha_y + \delta_y\,\eta.
\end{align}
This home-boost model has $L-1$ free parameters. At $L=2$ and $L=3$ the boost is an exact reparameterization of the free $\balpha^\tnr{h}$; the constraints appear for $L\ge4$, and we keep the boost throughout. Setting $\eta=0$ recovers the venue-neutral model \eqref{P.y.symmetry}. These smallest cases are familiar paired-comparison models: at $L=2$ the Bradley--Terry model \citep{Bradley52}, and at $L=3$ Davidson's ties model \citep{Davidson70}.

\paragraph{Points rules in practice.}
The count can be sufficient only if the rule is constant-sum, so that every match distributes the same total between its participants. Two familiar rules do not satisfy this: the football $(0\text{-}1\text{-}3)$ rule distributes $3$ points in a match with a winner but $2$ in a draw, and the \gls{nhl} $(0\text{-}1\text{-}2\text{-}2)$ rule distributes $2$ points in regulation time but $3$ in overtime. With $\xi_0=0$, for three outcomes the condition reads $0+\xi_2=2\xi_1$, so the draw is worth half a win and the rule is $(0\text{-}1\text{-}2)$; the half-point follows from the model rather than being imposed, and it is the value that \citet[Sec.~5.5]{Lasek18} identify only through simulations. For four outcomes it reads $0+\xi_3=\xi_1+\xi_2$, which leaves one interior value free, $\xi_2=\xi_3-\xi_1$, to be fit or posited. For example, the $(0\text{-}1\text{-}2\text{-}3)$ rule of the \gls{shl}, also postulated by \citet{LeBrun09} and \citet{Luszczyszyn17}, satisfies it.

\subsection{Relation to other models}\label{Sec:related.models}
Sufficiency is a property of the \gls{ac} model, so two sets of results that produce the same scores on the same schedule yield the same estimated skills. This means that the pattern of the outcomes across the opponents, which may matter in other models \citep{Remage66,deCani69,Tiwisina19}, does not matter here. Different models fitted to the same results also produce different rankings \citep{Ley19}.

The sufficiency of an accumulated score is a known theme in the \gls{irt} literature \citep{Andersen77,Masters82}, and Lemma~\ref{Lem:score.characterization} parallels \citet[Proposition 1]{Szczecinski26b}, where imposing the Elo (logistic) structure fixes the \gls{ac} form and its parameters. What is particular to the ranking problem is the home/away distinction.

Lemma~\ref{Lem:score.characterization} excludes the widely used \gls{cl} (proportional-odds) model \citep{Agresti92,Tutz20}, $\mfP_y(z)=\sigma(z-c_{y-1})-\sigma(z-c_y)$ with the logistic function $\sigma(\cd)$ and increasing thresholds $c_y$. There, the gradient is $\dot\ell_y(z)=-1+\sigma(c_{y-1}-z)+\sigma(c_y-z)$, in which the outcome enters through the thresholds, inside terms that also contain $z$. It therefore does not separate from the skill difference, and all matches must enter the estimation \citep[Sec.~8.2.2]{Agresti13_book}. The point count is thus not sufficient for the \gls{cl} model, except at $L=2$, where both reduce to Bradley--Terry \citep{Bradley52}.

For us the identification matters more than the sufficiency itself. The slopes of the \gls{ac} model, and hence the score-points, can be \emph{estimated} from the data, so the model uncovers the points rule that the outcomes imply.\footnote{The alternatives are to posit ground-truth skills and simulate \citep{Lasek18}, which still requires a model to be estimated, or to build a model on the order alone \citep{Szczecinski23b}, which entails complex combinatorial optimization.}


\section{Estimating the skills}\label{Sec:Estimation}

We now derive the estimate and give a condition under which the score order and the skill order agree.

\subsection{The estimation problem}\label{Sec:MAP}

We work with the \gls{ac} model of Sec.~\ref{Sec:AC.model} and the \gls{map} criterion \citep{Lasek18,Newman23}
\begin{align}
\label{MAP.objective}
    \hat\btheta = \argmax_{\btheta}
    \sum_{i,j}\sum_y
    k^\tnr{h}_{y;i,j}\, \ell^\tnr{h}_y(\theta_i-\theta_j)
     + \sum_i\log f(\theta_i),
\end{align}
where each match is counted once, from the host's perspective: $k^\tnr{h}_{y;i,j}$ is the number of matches in which $i$ hosts $j$ with $i$'s outcome $y$, associated with the home log-likelihood $\ell^\tnr{h}_y(z)=\alpha^\tnr{h}_y+\delta_y z-\log D^\tnr{h}(z)$ from \eqref{AC.model.ell}.

The term $\log f(\theta_i)$ in \eqref{MAP.objective} is the log-prior, which acts as the regularizer. Provided it decreases to $-\infty$ in both directions, it prevents the divergence of $\hat\theta_i$ that would otherwise occur when all outcomes of player $i$ are wins ($y=L-1$) or all are losses ($y=0$). Thus, the prior ensures that $\hat\btheta$ exists, which is not guaranteed by the \gls{ml} principle of \citet{Zermelo29}. 
We use the Gaussian prior of precision $\gamma$, which shrinks the estimates towards zero and, being strictly concave, makes $\hat\btheta$ unique, 
\begin{align}
\label{prior.ridge}
    \log  f(\theta_i) = -\tfrac{\gamma}{2}\,\theta_i^2 + \tnr{const}.
\end{align}

We set the derivative of \eqref{MAP.objective} with respect to $\theta_i$ to zero. The matches hosted by $i$ contribute $\dot\ell^\tnr{h}_y(\theta_i-\theta_j)$ and those $i$ visited contribute $-\dot\ell^\tnr{h}_{L-1-y}(\theta_j-\theta_i)=\dot\ell^\tnr{a}_y(\theta_i-\theta_j)$, both indexed by $i$'s own outcome $y$, see \eqref{k.a.def}. With the gradients \eqref{dot.ell_y(z).AC} and \eqref{dot.ell_y(z).AC.a} and the prior \eqref{prior.ridge}, this gives
\begin{align}
\label{MAP.venue}
    \sum_{j\neq i}\bigl[
    k^\tnr{h}_{i,j}\,G^\tnr{h}(\hat\theta_i-\hat\theta_j)
    + k^\tnr{a}_{i,j}\,G^\tnr{a}(\hat\theta_i-\hat\theta_j)\bigr]
    + \gamma\,\hat\theta_i
    = s_i,
\end{align}
where the data appear on the right alone, as the score \eqref{score.definition}, $s_i=\sum_y\delta_y\,k_{y;i}$ (the score-points being the slopes, $\xi_y=\delta_y$, by Lemma~\ref{Lem:score.characterization}), which does not distinguish the venues; the distinction appears only through $G^\tnr{h}(\cd)$ and $G^\tnr{a}(\cd)$.\footnote{Both summations are immediate, $\delta_y$ depending on neither the venue nor the opponent, and $G^\tnr{h}(\cd)$, $G^\tnr{a}(\cd)$ not depending on the outcome.}

Defining the average expected score and half the home-away gap as
\begin{align}
\label{G.split}
    \tilde G(\cd)=\tfrac12\bigl(G^\tnr{h}(\cd)+G^\tnr{a}(\cd)\bigr),\qquad \check G(\cd)=\tfrac12\bigl(G^\tnr{h}(\cd)-G^\tnr{a}(\cd)\bigr),
\end{align}
so that $G^\tnr{h}(\cd)=\tilde G(\cd)+\check G(\cd)$ and $G^\tnr{a}(\cd)=\tilde G(\cd)-\check G(\cd)$, \eqref{MAP.venue} may be written as 
\begin{align}
\label{MAP.split}
    \sum_{j\neq i}\bigl[
    k_{i,j}\,\tilde G(\hat\theta_i-\hat\theta_j)
    + d_{i,j}\,\check G(\hat\theta_i-\hat\theta_j)\bigr]
    + \gamma\,\hat\theta_i
    = s_i,
\end{align}
where $k_{i,j}=k^\tnr{h}_{i,j}+k^\tnr{a}_{i,j}$ is the number of matches between $i$ and $j$, and $d_{i,j}=k^\tnr{h}_{i,j}-k^\tnr{a}_{i,j}=-d_{j,i}$ is the venue imbalance of $i$ against $j$.

The outcomes enter \eqref{MAP.split} only through the score $s_i$, which is the converse direction of Lemma~\ref{Lem:score.characterization} and extends \citet{Zermelo29} and \citet[Sec.~8.1.4]{Agresti13_book}, both venue-neutral, to the home advantage. The individual fixtures are not needed: the schedule enters only through $k_{i,j}$ and $d_{i,j}$.

The term $d_{i,j}\check G(\cd)$ expresses the effect of the venue and disappears if (i) $i$ and $j$ meet equally often at each other's venues, so that $d_{i,j}=0$, or (ii) the model is venue-neutral \eqref{P.y.symmetry}, in which case $G^\tnr{h}(\cd)=G^\tnr{a}(\cd)$ and $\check G(\cd)\equiv0$.

\subsection{When the score recovers the skill order}\label{Sec:score.order}

Sufficiency guarantees that the scores $s_i$ contain all the information about the skills $\btheta$, but it does not tell us whether ranking by the scores reproduces the ranking by the \gls{map} skills $\hat\btheta$. Thus, following  \citet{Zermelo29}, we ask when the score order of the players in the pair $(i,j)$ reproduces their skill order, meaning the order of the estimates $\hat\theta_i$. We use the schedule vocabulary of Sec.~\ref{Sec:Model}: the pair $(i,j)$ is schedule-equivalent if $k_{i,l}=k_{j,l}$ for $l\ne i,j$, and the schedule is venue-balanced for the player $i$ if $i$ hosts and visits every opponent equally often, \ie $d_{i,l}=0$ for every $l$; these two conditions are independent.

\begin{proposition}[Ranking by score on a schedule-equivalent pair]\label{Prop:score.ranking.balanced}
If the pair $(i,j)$ is schedule-equivalent and the schedule is venue-balanced for $i$ and $j$, then their score and skill orders agree:
\begin{align}
\label{order.condition}
    s_i\ge s_j\iff\hat\theta_i\ge\hat\theta_j.
\end{align}
Under a venue-neutral (no home advantage) model, schedule equivalence alone guarantees \eqref{order.condition}.
\end{proposition}

\begin{proof}
We write $\hat z_{i,j}=\hat\theta_i - \hat\theta_j$ and subtract \eqref{MAP.split} for $j$ from that for $i$:
\begin{align}
\label{s.i.j.diff}
    s_i-s_j=\gamma\hat z_{i,j}
    &+\sum_l\bigl[k_{i,l}\tilde G(\hat z_{i,l})-k_{j,l}\tilde G(\hat z_{j,l})\bigr]
    +\sum_l\bigl[d_{i,l}\check G(\hat z_{i,l})-d_{j,l}\check G(\hat z_{j,l})\bigr].
\end{align}
The venue balance, $d_{i,l}=d_{j,l}=0$ for every $l$, and the venue-neutral model, $\check G(\cd)\equiv0$, each remove the second sum in \eqref{s.i.j.diff}.

The schedule equivalence, $k_{i,l}=k_{j,l}$ for $l\ne i,j$, yields
\begin{align*}
    s_i-s_j=\gamma\hat z_{i,j}+k_{i,j}\bigl[\tilde G(\hat z_{i,j})-\tilde G(\hat z_{j,i})\bigr]
    +\sum_{l\ne i,j}k_{i,l}\bigl[\tilde G(\hat z_{i,l})-\tilde G(\hat z_{j,l})\bigr].
\end{align*}
Because $\tilde G(\cd)$ is strictly increasing, each bracket has the sign of the difference of its arguments: $\hat z_{i,l}-\hat z_{j,l}=\hat z_{i,j}$ (for the opponent terms, where the $\hat\theta_l$ cancels) and $\hat z_{i,j}-\hat z_{j,i}=2\hat z_{i,j}$ (for the head-to-head). With $\gamma>0$ and nonnegative counts, every term has the sign of $\hat{z}_{i,j}$, which proves \eqref{order.condition}. 
\end{proof}

The argument needs the venue balance itself: equal imbalances against the other opponents, $d_{i,l}=d_{j,l}$ for $l\ne i,j$, are not enough, because, unlike $\tilde G(\cd)$, the half-gap $\check G(\cd)$ is not monotone,\footnote{Since $G^\tnr{h}(z)=1-G^\tnr{a}(-z)$, the half-gap is even, $\check G(-z)=\check G(z)$.} so equal but nonzero imbalances leave a residual $\sum_{l\ne i,j} d_{i,l}[\check G(\hat z_{i,l})-\check G(\hat z_{j,l})]$ that need not have the sign of $\hat z_{i,j}$.

\begin{corollary}[Round-robin]\label{Cor:round.robin}
On a round-robin every pair is schedule-equivalent, so under a venue-neutral model Proposition~\ref{Prop:score.ranking.balanced} applies to all pairs at once and \eqref{order.condition} holds for every $i,j=1,\ld,M$. With a home advantage the schedule must be venue-balanced as well, which requires the meetings split evenly, $k^\tnr{h}_{i,l}=k^\tnr{a}_{i,l}=k/2$ ($k$ must be even); the smallest such venue-balanced round-robin is the double round-robin. In both cases point counting and the exact \gls{map} estimate induce identical rankings, ties included.
\end{corollary}

For binary win/loss outcomes, and ignoring the venue, this is the regular-tournament theorem of \citet{Zermelo29} and \citet{Ford57}, restated here for ordinal outcomes with the home/away venue and the regularized estimate.\footnote{\citet{Zermelo29} adopted $\delta_1=\tfrac12$ for ternary ($L=3$) matches as a convention, implemented by the fiction that twice as many games were played, a win counting double and a draw single; he notes that this reproduces the conventional half-point. The fiction is needed because his model is binary, whereas we show $\delta_1=\tfrac12$ to be a consequence of sufficiency.}

\begin{corollary}[Order invariance]\label{Cor:order.invariance}
For every pair $(i,j)$ to which Proposition~\ref{Prop:score.ranking.balanced} applies, the skill order of $i$ and $j$, which is their score order, does not depend on
(i) the choice of a strictly increasing $\tilde G(\cd)$, and hence on the intercepts $\balpha$ and, under the home-boost model, on the home advantage;
(ii) the replacement of the ridge~\eqref{prior.ridge} by any log-concave prior $\prod_i f(\theta_i)$; or
(iii) the presence of the prior, the \gls{ml} estimate ($\gamma=0$) giving the same order whenever it is finite.
\end{corollary}

\begin{proof}
All three parts follow from the proof of Proposition~\ref{Prop:score.ranking.balanced}, in which the sign of $s_i-s_j$ in \eqref{s.i.j.diff} rests only on $\tilde G(\cd)$ being strictly increasing and on the log-prior term having the sign of $\hat z_{i,j}$. The particular $\tilde G(\cd)$ is never used, and $\balpha$ and the home advantage enter only through it while preserving its monotonicity, which gives~(i). A log-concave prior replaces $\gamma\hat z_{i,j}$ by $(\log f)'(\hat\theta_j)-(\log f)'(\hat\theta_i)$, which has the sign of $\hat z_{i,j}$ because $(\log f)'(\cd)$ is non-increasing, every other term being unchanged, which gives~(ii). At $\gamma=0$ the prior term disappears and the remaining terms still have the sign of $\hat z_{i,j}$, which gives~(iii); regularization only guarantees that $\hat\btheta$ exists, whence the restriction to schedules on which the \gls{ml} estimate is finite.
\end{proof}

The Elo rating is one case of practical interest. Its update targets the optimality condition \eqref{MAP.split} with its own increasing $\tilde G(\cd)$, the logistic function, and without a prior \citep{Szczecinski26b}, so wherever Proposition~\ref{Prop:score.ranking.balanced} applies it induces the same order as the \gls{map} estimate and as point counting. Where it does not apply, the orders differ. This underlies the criticism of the UEFA club coefficients, which award the same points for every win regardless of the opponent, and the proposal to replace them by an Elo-based formula \citep{Csato24,Lapre25}.

\section{Numerical examples}\label{Sec:Numerical}

Real leagues are not always round-robins, so Corollary~\ref{Cor:round.robin} need not apply. However, Proposition~\ref{Prop:score.ranking.balanced} applies for every pair that is schedule-equivalent and for which the schedule is venue-balanced. Only the remaining pairs can be reordered, and we measure how much that matters in practice.

\subsection{Data and estimation}\label{Sec:num.estimation}

\paragraph{Data.}
We consider (i) association football, from Division~1 (England, top tier before 1992/93), the \gls{epl} (England, top tier since 1992/93), the Championship (England, second tier), and the Bundesliga (Germany, top tier); (ii) ice hockey, from the \gls{nhl} (North America) and the \gls{shl} (Sweden); and (iii) volleyball, from SuperLega (Italy, top tier). In football, $L=3$ and the outcomes are loss ($y=0$), draw ($y=1$), and win ($y=2$); Division~1 and the Bundesliga cover the two-point and the three-point eras, whose points rules we compare in Sec.~\ref{Sec:num.points}. In ice hockey, $L=4$ and we distinguish regulation loss ($y=0$), overtime loss ($y=1$), overtime win ($y=2$), and regulation win ($y=3$).\footnote{The four outcomes admit no draw, which decides the seasons we exclude. Draws occur in the \gls{nhl} before 2005/06, and in the \gls{shl} before 1999/00 and in 2004/05--2009/10.} In SuperLega, $L=6$ and $y=0,\ld,5$ are the set scores $(0\tnr{-}3,\ld,3\tnr{-}0)$.
\paragraph{Estimation.}
We use the \gls{ac} model \eqref{AC.model} with constant-sum slopes,
$\delta_{L-1-y}+\delta_y=1$, which is the result of Lemma~\ref{Lem:score.characterization}. This leaves $\lfloor (L-2)/2\rfloor$ slopes to estimate: none in football, where $\delta_1=\tfrac12$ is determined, $\delta_1$ in ice hockey, and $\delta_1$ and $\delta_2$ in SuperLega.

We estimate the model parameters and the prior precision $\gamma$ jointly over the retained seasons by maximizing the marginal likelihood, in which the skills $\btheta$, assumed independent between seasons, are integrated out. We use the \gls{em} algorithm \citep[Ch.~11.2]{Barber12_Book}, treating $\btheta$ as missing data with posterior approximated as Gaussian $\mcN(\hat\btheta,\bH^{-1})$, where $\hat\btheta$ is the \gls{map} estimate \eqref{MAP.objective}.\footnote{$\bH$ is minus the Hessian of the objective in \eqref{MAP.objective} at $\hat\btheta$. With the Gaussian approximation, the expected log-likelihood depends on the remaining parameters in closed form. Since the likelihood depends on the skills only through the differences $z_{i,j}$, only their univariate marginals are needed, $z_{i,j}\sim\mcN(\hat z_{i,j},v_{i,j})$ with $v_{i,j}=[\bH^{-1}]_{i,i}+[\bH^{-1}]_{j,j}-2[\bH^{-1}]_{i,j}$, so the expectation is a one-dimensional integral per match, which we evaluate by Gauss-Hermite quadrature. The M-step then maximizes it over $\balpha^\tnr{h}$, $\bdelta$ and $\eta$, while the precision has the closed form $\hat\gamma^{-1}=\operatorname{mean}_i(\hat\theta_i^2+[\bH^{-1}]_{i,i})$, the second moment of the same Gaussian. The two steps are iterated until convergence.} The slopes are estimated in every case; for the intercepts we use the \emph{home-boost} model \eqref{alpha.sym->alpha} with $\eta$ free. Table~\ref{tab:params} shows the estimates except for the intercepts $\hat\balpha^\tnr{h}$ as they do not affect the ranking (Sec.~\ref{Sec:score.order}). Every league is also refitted with $\eta=0$, which is the venue-neutral case \eqref{P.y.symmetry}. We do not report those estimates, as the refit moves $\hat\delta_y$ by at most $0.4$ standard errors. We use it to measure the effect of the venue on the ranking (Table~\ref{tab:schedule}).

\begin{table}[t]
\centering
\caption{Estimated \gls{ac} parameters: the free slopes $\hat\delta_y$ and the home-advantage parameter $\hat\eta$, in leagues identified by their points rule $\bxi$ and over the seasons indicated by their start year. Uncertainties are model-based standard errors, obtained from the curvature of the marginal likelihood at the estimate.}
\label{tab:params}
{\footnotesize\setlength{\tabcolsep}{4pt}
\begin{tabular}{llcc}
\hline
League ($\bxi$) & Seasons & free $\hat\delta_y$ & $\hat\eta$\\
\hline
Division~1 $(0\text{-}1\text{-}2)$ & 1950--1980 & none & $0.80\pm0.02$\\
Division~1 $(0\text{-}1\text{-}3)$ & 1981--1991 & none & $0.72\pm0.04$\\
EPL $(0\text{-}1\text{-}3)$ & 1992--2025$^\tnr{a}$ & none & $0.58\pm0.02$\\
Championship $(0\text{-}1\text{-}3)$ & 1993--2025$^\tnr{a}$ & none & $0.51\pm0.02$\\
Bundesliga $(0\text{-}1\text{-}2)$ & 1965--1994$^\tnr{b}$ & none & $1.13\pm0.03$\\
Bundesliga $(0\text{-}1\text{-}3)$ & 1995--2025$^\tnr{a}$ & none & $0.56\pm0.03$\\
NHL $(0\text{-}1\text{-}2\text{-}2)$ & 2005--2025$^{\tnr{a,c}}$ & $0.38\pm0.02$ & $0.23\pm0.02$\\
SHL $(0\text{-}1\text{-}2\text{-}3)$ & \begin{tabular}{@{}l@{}}1999--2003,\\2010--2025$^\tnr{a}$\end{tabular} & $0.39\pm0.04$ & $0.42\pm0.03$\\
SuperLega $(0\text{-}0\text{-}1\text{-}2\text{-}3\text{-}3)$ & 2009--2024$^\tnr{a}$ & $0.20\pm0.02$, $0.40\pm0.02$ & $0.64\pm0.08$\\
\hline
\multicolumn{4}{l}{\footnotesize $^\tnr{a}$Excluding the COVID-affected seasons 2019/20 and 2020/21.}\\
\multicolumn{4}{l}{\footnotesize $^\tnr{b}$Excluding the twenty-team 1991/92 season.}\\
\multicolumn{4}{l}{\footnotesize $^\tnr{c}$Excluding the lockout-shortened 2012/13 season.}\\\end{tabular}
}
\end{table}

\subsection{Agreement of the score and skill orders}\label{Sec:num.ranking}

Table~\ref{tab:schedule} groups the seasons by schedule structure and reports, for each group, the meeting counts $k_{i,j}$, the venue balance and the share of schedule-equivalent pairs, which together determine the pairs covered by Proposition~\ref{Prop:score.ranking.balanced}, and how many of the pairs with distinct scores are reordered. Here the score \eqref{score.definition} is formed with the fitted slopes, $s_i=\sum_y\hat\delta_y k_{y;i}$, which Lemma~\ref{Lem:score.characterization} makes sufficient for the fitted model; the league's own rule is analyzed in Sec.~\ref{Sec:num.points}.

\begin{table}[!b]
\centering
\caption{Schedule structure and the agreement of the score and skill orders.
$k_{i,j}$ lists the meeting counts in the group; ``venue-bal.'' states whether every team
hosts and visits every opponent equally often; ``sched.-equiv.'' is the share of
schedule-equivalent pairs, the two together giving the pairs covered by
Proposition~\ref{Prop:score.ranking.balanced}. ``Pairs'' counts those with distinct
scores, and the last two columns how many of them are reordered, at the fitted $\hat\eta$,
with the rate in parentheses, and after a refit with $\eta=0$. Seasons excluded as in Table~\ref{tab:params}.}
\label{tab:schedule}
{\footnotesize\setlength{\tabcolsep}{4pt}
\begin{tabular}{llccccc}
\hline
 & & & & & \multicolumn{2}{c}{reordered}\\
\cline{6-7}
League (era) & $k_{i,j}$ & venue-bal. & sched.-equiv. & pairs & $\hat\eta$ & $\eta=0$\\
\hline
football, SuperLega  & $2$       & yes & $100\%$ & $33254$ & $0$ & $0$\\
\gls{shl} 2015--2025 & $4$       & yes & $100\%$ & $816$   & $0$  & $0$\\
\gls{shl} 2010--2014 & $5$       & no  & $100\%$ & $327$   & $0$  & $0$\\
\gls{shl} 1999--2003 & $4,6$     & yes & $27\%$  & $330$   & $6\,(1.8\%)$  & $8$\\
\gls{nhl} 2005--2007 & $0,1,4,8$ & no  & $14\%$  & $1298$  & $42\,(3.2\%)$ & $42$\\
\gls{nhl} 2008--2011 & $1,2,4,6$ & no  & $1\%$   & $1734$  & $28\,(1.6\%)$ & $27$\\
\gls{nhl} 2013--2018 & $2,3,4,5$ & no  & $6\%$   & $2662$  & $29\,(1.1\%)$ & $28$\\
\gls{nhl} 2021--2025 & $2,3,4$   & no  & $1\%$   & $2476$  & $25\,(1.0\%)$ & $24$\\
\hline
\end{tabular}}
\end{table}

By Corollary~\ref{Cor:round.robin}, there cannot be any reordering under the
venue-balanced round-robins (football, SuperLega and the \gls{shl} from 2015).
Table~\ref{tab:schedule} confirms this and shows three things beyond it.

\begin{itemize}
\item No schedule-equivalent pair is reordered, venue-balanced or not, so the venue
condition of Proposition~\ref{Prop:score.ranking.balanced} is not needed in these data.
The \gls{shl} seasons of 2010--2014 are schedule-equivalent and not venue-balanced, and
none of their pairs is reordered. Pair by pair, across the \gls{nhl} and the \gls{shl},
$719$ pairs are schedule-equivalent but not venue-balanced (not shown in
Table~\ref{tab:schedule}), and none is reordered.

\item Where schedule equivalence fails, few pairs are reordered, no group exceeding
$3.2\%$ of those with distinct scores.

\item Pairs are reordered more often as the range of $k_{i,j}$ widens. In the \gls{shl} of
1999--2003 the schedule is venue-balanced while the meeting counts differ, and the
reorderings appear. Across the four \gls{nhl} formats the proportion rises from
$1.0\%$ at the narrowest range, 2021--2025 with $k_{i,j}\in\set{2,3,4}$, to $3.2\%$ at the
widest, 2005--2007 with $k_{i,j}\in\set{0,1,4,8}$, \ie some pairs never meet while others
meet eight times. The share of schedule-equivalent pairs does not predict it, 2013--2018
having $6\%$ against $1\%$ in 2008--2011 and yet reordering $1.1\%$ against $1.6\%$.

\end{itemize}

The fits add two further points, neither of them in the table.
\begin{itemize}
\item Of the $1080$ pairs finishing with the same score, the fit assigns equal skills to
the $1052$ that Proposition~\ref{Prop:score.ranking.balanced} covers and separates the $28$
that it does not. In the \gls{shl} seasons of 2010--2014, which violate the venue-balance condition of
Proposition~\ref{Prop:score.ranking.balanced}, all three ties are broken while no strict
order is disturbed.

\item The reorderings happen between teams that have similar scores. In the \gls{nhl}
the reordered pairs differ by a median of $0.4$ score-points (one overtime loss in place of a
regulation loss, worth $\hat\delta_1=0.38$) against $7.0$ for all pairs with distinct scores.
\end{itemize}

\subsection{Points rules as models}\label{Sec:num.points}

By Lemma~\ref{Lem:score.characterization} a league using a constant-sum points rule and ranking by it adopts
(implicitly) an \gls{ac} model whose slopes $\bdelta$ are the normalized score-points
$\bxi$. The \emph{uniform} rule $\xi_y=y$, which awards one point per level, is the simplest
constant-sum rule and is sufficient for its own \gls{ac} model.
Table~\ref{tab:points} reports the agreement between the skill and the score orders, by
Kendall's $\tau_b$, which penalizes tied pairs, averaged across seasons, for the league's own
rule, the uniform rule and the fitted slopes $\hat\bdelta$. It gives five readings.

\begin{table}[t]
\centering
\caption{Points rules and the ranking. Const.-sum=``yes'' is true when the league's rule
$\bxi$ satisfies $\xi_y+\xi_{L-1-y}=\tnr{const}$. The last three columns are averaged
over the seasons of Table~\ref{tab:params}, with the \gls{shl} split by era. An entry of $1$ means that the score reproduces the order
\emph{exactly}, under Corollary~\ref{Cor:round.robin} on a venue-balanced
round-robin.}
\label{tab:points}
{\footnotesize\setlength{\tabcolsep}{4pt}
\begin{tabular}{lcccc}
\hline
 & & \multicolumn{3}{c}{$\bar\tau$ with the skill order}\\
\cline{3-5}
League ($\bxi$) & const.-sum & own rule & uniform & fitted\\
\hline
Division~1 $(0\text{-}1\text{-}2)$& yes & \multicolumn{3}{c}{\rule[.45ex]{5.2em}{.4pt}\;$1$\;\rule[.45ex]{5.2em}{.4pt}}\\
Division~1 $(0\text{-}1\text{-}3)$& no  & $0.962$ & \multicolumn{2}{c}{\rule[.45ex]{2.8em}{.4pt}\;$1$\;\rule[.45ex]{2.8em}{.4pt}}\\
EPL $(0\text{-}1\text{-}3)$ & no  & $0.965$ & \multicolumn{2}{c}{\rule[.45ex]{2.8em}{.4pt}\;$1$\;\rule[.45ex]{2.8em}{.4pt}}\\
Championship $(0\text{-}1\text{-}3)$ & no  & $0.952$ & \multicolumn{2}{c}{\rule[.45ex]{2.8em}{.4pt}\;$1$\;\rule[.45ex]{2.8em}{.4pt}}\\
Bundesliga $(0\text{-}1\text{-}2)$& yes & \multicolumn{3}{c}{\rule[.45ex]{5.2em}{.4pt}\;$1$\;\rule[.45ex]{5.2em}{.4pt}}\\
Bundesliga $(0\text{-}1\text{-}3)$& no  & $0.962$ & \multicolumn{2}{c}{\rule[.45ex]{2.8em}{.4pt}\;$1$\;\rule[.45ex]{2.8em}{.4pt}}\\
NHL $(0\text{-}1\text{-}2\text{-}2)$ & no  & $0.907$ & $0.965$ & $0.968$\\
SHL 1999--2003 $(0\text{-}1\text{-}2\text{-}3)$ & yes & \multicolumn{2}{c}{\rule[.45ex]{2.6em}{.4pt}\;$0.970$\;\rule[.45ex]{2.6em}{.4pt}} & $0.964$\\
SHL 2010--2014 $(0\text{-}1\text{-}2\text{-}3)$ & yes & \multicolumn{2}{c}{\rule[.45ex]{2.6em}{.4pt}\;$0.986$\;\rule[.45ex]{2.6em}{.4pt}} & $0.995$\\
SHL 2015--2025 $(0\text{-}1\text{-}2\text{-}3)$ & yes & \multicolumn{2}{c}{\rule[.45ex]{2.6em}{.4pt}\;$0.990$\;\rule[.45ex]{2.6em}{.4pt}} & $1$\\
SuperLega $(0\text{-}0\text{-}1\text{-}2\text{-}3\text{-}3)$ & yes & $0.963$ & $0.993$ & $1$\\
\hline
\end{tabular}}
\end{table}

\begin{itemize}
\item On average, the uniform rule outperforms the league's own rule in every league
where the two differ.

\item The last column restates Table~\ref{tab:schedule}, since both compare the score
formed with the fitted slopes against the skill order. They treat ties differently.
Table~\ref{tab:schedule} counts only the reordered pairs, while $\tau_b$ also registers a
pair whose scores are tied and whose fitted skills are not. In 2010--2014 the schedule is
not venue-balanced, so Proposition~\ref{Prop:score.ranking.balanced} does not apply, the
fit separates the equally scored pairs and $\tau_b$ stops at $0.995$ although nothing is
reordered. From 2015 on Proposition~\ref{Prop:score.ranking.balanced} applies, the equally scored pairs receive equal skills and $\tau_b=1$.

\item A rule that is not constant-sum has no guarantee of exactness. The two-point
football rule is constant-sum, so Corollary~\ref{Cor:round.robin} makes its score
reproduce the skill order exactly on the double round-robin. The football
$(0\text{-}1\text{-}3)$ and \gls{nhl} $(0\text{-}1\text{-}2\text{-}2)$ rules are not, the
football rule reaching $0.952$ to $0.965$ and the \gls{nhl} rule $0.907$. Even the uniform
rule reaches only $0.965$ in the \gls{nhl}, its schedule being the furthest from
schedule-equivalent in Table~\ref{tab:schedule}.

\item SuperLega's rule is constant-sum, so its score is sufficient, and the uniform rule
still outperforms it. The estimates space its six outcomes almost evenly, while its own rule gives
equal score-points to the two worst and to the two best, $\xi_0=\xi_1$ and
$\xi_4=\xi_5$, which the estimates separate by $0.20$ each. Its score is therefore
sufficient for a model the data do not support, and constant-sum is a necessary condition
on the points rule and not a guarantee of its quality.

\item The fitted slopes and the uniform rule give almost the same agreement, the largest
difference in Table~\ref{tab:points} being $0.010$, and the fitted slopes are the higher
wherever the two differ apart from the \gls{shl} of 1999--2003. The fitted slopes are
themselves nearly uniform. In the \gls{nhl} $\hat\delta_1=0.38\pm0.02$ lies just over two standard
errors from $\tfrac13$, and in SuperLega the two vectors differ by $0.004$ in the interior
pair alone, which Table~\ref{tab:params} rounds away.

\end{itemize}

The uniform rule is thus the one to adopt when the aim is to rank by skill. It is constant-sum at every $L$, so it is
sufficient for its own \gls{ac} model, it requires no estimation, and in each of the nine
leagues of Table~\ref{tab:params} it outperforms the league's own rule unless that rule is
already the uniform one, the fitted slopes adding at most $0.01$ beyond it. The \gls{nhl} has the most to gain,
its agreement rising from $0.907$ to $0.965$, while the two-point football rule and the
\gls{shl}'s rule are the uniform rule already.

\section{Conclusions}\label{Sec:Conclusions}

Rankings in sport leagues are produced by counting points, while the statistical literature ranks by fitting an outcome model and sorting the estimated skills. We have asked how the two are related, taking the practice as given and looking for the model in which the accumulated score is the right summary of the data. Our findings are the following:

\begin{itemize}
\item The point count is a sufficient statistic for the skills if and only if the score-points are constant-sum and the outcomes follow the \gls{ac} model, whose slopes are then the score-points themselves (Lemma~\ref{Lem:score.characterization}). The football $(0\text{-}1\text{-}3)$ and \gls{nhl} $(0\text{-}1\text{-}2\text{-}2)$ rules are not constant-sum, so their scores are not sufficient statistics, while those of the \gls{shl} and SuperLega are. The model also fixes what was previously a free convention: the draw in a three-outcome sport is worth half a win.

\item The score and skill orders agree for a pair of players that meets every other opponent equally often and plays as often at home as away (Proposition~\ref{Prop:score.ranking.balanced}); this does not depend on the intercepts, on the home advantage or on the prior (Corollary~\ref{Cor:order.invariance}) and applies to all players on a venue-balanced round-robin (Corollary~\ref{Cor:round.robin}). Point counting is then the exact ranking under the fitted model, and it coincides with the Elo ranking \citep{Szczecinski26b}. This extends the theorem of \citet{Zermelo29} and \citet{Ford57} to ordinal outcomes and the venue.

\item In the nine leagues the two orders differ only where schedule equivalence fails, and the reorderings follow the meeting counts rather than the venue, reaching at most $3.2\%$ of the pairs with distinct scores in the \gls{nhl}. They occur between teams that have similar scores.

\item Constant-sum is necessary for the score to be sufficient, and SuperLega shows that it does not guarantee a good ranking. Its rule is constant-sum yet performs worse than the uniform rule $\xi_y=y$. We thus recommend the latter to a league whose aim is to rank by skill. It is constant-sum for all $L$, it needs no estimation, and it outperforms the league's own rule unless that rule is already the uniform one, \eg in the \gls{shl}. The \gls{nhl} would gain the most from it.

\end{itemize}

\citet{Tiwisina19} report empirical results consistent with these findings. On the double round-robin of the Bundesliga, counting points matches or outperforms the Elo and \gls{lop} rankings, the more elaborate schemes gaining only in the \gls{nfl} and in tennis, where the pairs do not meet equally often. Separating the two football rules on the same seasons, they find $(0\text{-}1\text{-}2)$ to explain the results about $10\%$ better than $(0\text{-}1\text{-}3)$, which agrees with Sec.~\ref{Sec:num.points} under a different criterion.

A similar analysis may be applied to other point-counting rules, to interpret them and to understand their limitations. In particular, we may let the score-points depend on the venue, so that the outcome is worth $\xi^\tnr{h}_y$ to the host and $\xi^\tnr{a}_y$ to the visitor. The score remains sufficient in that case, the two rules being linked by the reciprocity \eqref{P.h.P.a.reciprocity}, and constant-sum reduces them to the single rule that leagues use. The main limitation is in the statistical premise, because players are aware of the points rule, which then acts as an incentive. For example, the $(0\text{-}1\text{-}3)$ rule, explicitly introduced to reward offensive play, does not need to reflect the probabilistic model. Analyzing that incentive would require different tools.

\appendix
\section{Proof of Lemma~\ref{Lem:score.characterization}}\label{App:proof}

Throughout the proof the score-points are normalized so that $\xi_0=0$ and $\xi_{L-1}=1$, as in Sec.~\ref{Sec:AC.model}. For every player $i$ and at any values of the remaining skills $\btheta_{-i}$, which are nuisance parameters, the score $s_i=\sum_y\xi_y k_{y;i}$ is sufficient for the skill $\theta_i$ if and only if the log-likelihood factorizes as
\begin{align}
\label{factorization}
    \mathcal L = Q(s_i,\btheta)+R(\bk,\btheta_{-i}),
\end{align}
where $\bk=\set{k^\tnr{h}_{y;i,j}}$ collects the counts of every outcome and ordered pair, and $R$ does not involve $\theta_i$. Lemma~\ref{Lem:score.characterization} requires \eqref{factorization} at every schedule. The second term has zero gradient with respect to $\theta_i$, so \eqref{factorization} holds if and only if the counts enter $g_i=\partial\mathcal L/\partial\theta_i$ only through $s_i$ (integration in $\theta_i$ giving the converse). Here the log-likelihood counts each match once, from its host,
\begin{align}
    \mathcal L=\sum_{i,j}\sum_y k^\tnr{h}_{y;i,j}\,\ell^\tnr{h}_y(z_{i,j}),\qquad z_{i,j}=\theta_i-\theta_j.
\end{align}
The gradient is assembled as in Sec.~\ref{Sec:MAP}, here with $\dot\ell^\tnr{h}_y(\cd)$ left unspecified and without the prior,
\begin{align}
\label{Gi.gradient}
    g_i=\frac{\partial\mathcal L}{\partial\theta_i}
    =\sum_{j}\sum_y \bigl[k^\tnr{h}_{y;i,j}\,\dot\ell^\tnr{h}_y(z_{i,j})
    +k^\tnr{a}_{y;i,j}\,\dot\ell^\tnr{a}_y(z_{i,j})\bigr].
\end{align}
Since the outcome model is common to all pairs, it is enough to impose this requirement at any $i$. Keeping $s_i$ fixed, we vary the counts of the matches between $i$ and any $j$ at home, we use the home/away reciprocity, and we vary the per-venue scores; together they force conditions (i) and (ii) of Lemma~\ref{Lem:score.characterization}. The converse, that a model satisfying (i) and (ii) makes the score sufficient, is treated at the end.

First, change the home counts of $i$ against any $j$ as $\tilde k^\tnr{h}_{y;i,j}=k^\tnr{h}_{y;i,j}+\epsilon_y$, with $\epsilon_y$ keeping $\tilde k^\tnr{h}_{y;i,j}\ge 0$. Keeping constant the number of these matches ($\sum_y \epsilon_y=0$) and the score they produce ($\sum_y\xi_y \epsilon_y=0$) makes $\bepsilon$ orthogonal to $[1,\ld,1]$ and to $\bxi$. The counts being integers, only integer $\bepsilon$ are available. However, $\bxi$ is rational, so the two equations have a rational basis of solutions (which becomes integer once the denominators are cleared). Since sufficiency is required at every schedule, we may use one where $i$ hosts $j$ enough times for each such vector to keep the counts non-negative, so the admissible $\bepsilon$ span the orthogonal complement of $\tnr{span}\set{[1,\ld,1],\bxi}$ over the reals. Invariance of \eqref{Gi.gradient} requires one orthogonality more, $\sum_y \epsilon_y\,\dot\ell^\tnr{h}_y(z_{i,j})=0$, so every admissible $\bepsilon$ is orthogonal to $\tnr{span}\set{[1,\ld,1],\bxi,[\dot\ell^\tnr{h}_0(z_{i,j}),\ld,\dot\ell^\tnr{h}_{L-1}(z_{i,j})]}$. The $\bepsilon$ span $L-2$ dimensions, so that span is at most two-dimensional, and it holds the independent vectors $[1,\ld,1]$ and $\bxi$, hence equals their span and the vector of derivatives lies in it. As $z_{i,j}$ is arbitrary, this holds at every $z$:
\begin{align}
\label{affine.gradient.h}
    \dot\ell^\tnr{h}_y(z)=c^\tnr{h}(z)\,\xi_y+u^\tnr{h}(z).
\end{align}
The same change applied to the away counts of $i$ gives $\dot\ell^\tnr{a}_y(z)=c^\tnr{a}(z)\,\xi_y+u^\tnr{a}(z)$.

Second, the reciprocity links the two venues: $\dot\ell^\tnr{a}_y(z)=-\dot\ell^\tnr{h}_{L-1-y}(-z)$, which, using \eqref{affine.gradient.h} on both sides, reads
\begin{align}
\label{reciprocity.affine}
    c^\tnr{a}(z)\,\xi_y+c^\tnr{h}(-z)\,\xi_{L-1-y}=-u^\tnr{a}(z)-u^\tnr{h}(-z)=:w(z).
\end{align}
Writing \eqref{reciprocity.affine} at $y$ and at $L-1-y$ and subtracting gives $[c^\tnr{a}(z)-c^\tnr{h}(-z)](\xi_y-\xi_{L-1-y})=0$; at $y=0$ the second factor is $\xi_0-\xi_{L-1}=-1$, so $c^\tnr{a}(z)=c^\tnr{h}(-z)$ for all $z$, and the two vanish together. Substituting back, $c^\tnr{a}(z)\,(\xi_y+\xi_{L-1-y})=w(z)$ for all $y$, so either $c^\tnr{a}(\cd)\equiv0$, or the sum $\xi_y+\xi_{L-1-y}$ is constant in $y$, and the normalization fixes the constant at $\xi_0+\xi_{L-1}=1$, which is~(i). The first case is excluded, because $c^\tnr{h}(\cd)\equiv0$ makes $\dot\ell^\tnr{h}_y(z)$ independent of $y$, so $\ell^\tnr{h}_y(z)=A(z)+B_y$, and $\sum_y\mfP^\tnr{h}_y(z)=1$ then forces $A(z)$ to be constant, contradicting the non-triviality of the model. Condition~(i) comes from the swap of the two players, $y\mapsto L-1-y$ and $z\mapsto-z$, rather than from the venue. A venue-neutral model gives the same equation and the same conclusion.

Third, 
substituting \eqref{affine.gradient.h} and its away counterpart into \eqref{Gi.gradient} yields
\begin{align}
\label{gradient.split}
    g_i=\sum_j \bigl[c^\tnr{h}(z_{i,j})\,s^{\tnr{h},(j)}_i+c^\tnr{a}(z_{i,j})\,s^{\tnr{a},(j)}_i\bigr]
    +\sum_j \bigl[u^\tnr{h}(z_{i,j})\,k^\tnr{h}_{i,j}+u^\tnr{a}(z_{i,j})\,k^\tnr{a}_{i,j}\bigr],
\end{align}
where $s^{\tnr{h},(j)}_i=\sum_y\xi_y k^\tnr{h}_{y;i,j}$ and $s^{\tnr{a},(j)}_i=\sum_y\xi_y k^\tnr{a}_{y;i,j}$ are the per-venue, per-opponent scores, and the second sum is free of the outcomes. These scores add up to $s_i$, so writing one of them, say $s^{\tnr{h},(j_0)}_i$, as $s_i$ less all the others turns the first sum in \eqref{gradient.split} into $c^\tnr{h}(z_{i,j_0})\,s_i$ plus each remaining per-venue score weighted by its own $c^\tnr{h}(z_{i,j})$ or $c^\tnr{a}(z_{i,j})$ less $c^\tnr{h}(z_{i,j_0})$. On a schedule with enough matches against each opponent, 
those scores still vary at fixed $s_i$, so every weight vanishes and $c^\tnr{h}(z_{i,j})=c^\tnr{a}(z_{i,j})=c^\tnr{h}(z_{i,j_0})$
. As the skill differences are arbitrary, $c^\tnr{h}(z)=c^\tnr{a}(z)=c$, a constant, and $c\neq0$ by the non-triviality argument above.

Integrating $\dot\ell^\tnr{h}_y(z)=c\,\xi_y+u^\tnr{h}(z)$ in $z$ and imposing $\sum_y\mfP^\tnr{h}_y(z)=1$ gives $\mfP^\tnr{h}_y(z)=\e^{\alpha^\tnr{h}_y+c\,\xi_y z}/\sum_l\e^{\alpha^\tnr{h}_l+c\,\xi_l z}$, with the intercepts $\alpha^\tnr{h}_y$ arising as integration constants, hence unconstrained. Requiring that higher skills make higher outcomes more probable gives $c>0$, and absorbing $c$ into the scale of the skills yields~(ii). The away model is then $\mfP^\tnr{a}_y(z)\propto\e^{\alpha^\tnr{h}_{L-1-y}+(\xi_y-1)z}$; absorbing the factor $\e^{-z}$, which does not depend on $y$, into the normalization leaves the slopes $\xi_y$ unchanged, so the two venues differ only in the intercepts, $\alpha^\tnr{a}_y=\alpha^\tnr{h}_{L-1-y}$. Writing the slopes as $\delta_y=\xi_y$, as in the main text, gives the form \eqref{AC.model}.

Conversely, Sec.~\ref{Sec:MAP} shows that for a model satisfying (i) and (ii) the outcomes enter \eqref{Gi.gradient} only through $s_i$, so \eqref{factorization} holds and the score is sufficient. The changes used above are therefore enough. They force the \gls{ac} form, which is invariant under every score-preserving change, so no other change constrains the model.

\begingroup\setstretch{1}

\endgroup

\end{document}